\documentclass[11pt]{article}
\usepackage[T1]{fontenc}
\usepackage{lmodern}
\usepackage{amsmath,amssymb,amsthm}
\usepackage[margin=30mm]{geometry}
\usepackage[colorlinks=true,linkcolor=blue,citecolor=blue,urlcolor=blue]{hyperref}

\newcommand{\SL}{\mathrm{SL}}
\newtheorem{theorem}{Theorem}
\newtheorem{lemma}{Lemma}
\newtheorem{corollary}{Corollary}
\newtheorem{proposition}{Proposition}
\theoremstyle{remark}
\newtheorem*{remark}{Remark}

\newcommand{\supp}{\operatorname{supp}}
\title{Local spectral gap and the Stuck--Zimmer conjecture}
\author{Simon Machado\footnote{ETH Zurich, Ramistrasse 101, 8006 Zurich; \texttt{smachado@ethz.ch}}, Yuval Yifrach\footnote{University of Zurich, Winterthurerstrasse 190, 8057 Zürich, Switzerland; \texttt{yuval.yifrach@math.uzh.ch}} }

\date{\today}
\begin{document}
\maketitle

\begin{abstract}
We prove two new cases of the Stuck--Zimmer conjecture. 

First, we settle the Stuck--Zimmer conjecture for ergodic probability
preserving actions of irreducible lattices (both uniform and non-uniform) in connected semisimple
real Lie groups with finite center, no compact factor and real
rank at least two. We furthermore extend the Stuck--Zimmer theorem
to all irreducible probability preserving actions of such groups
whenever one simple factor is locally isomorphic to $\SL_2(\mathbb R)$. In particular, this settles the Stuck--Zimmer conjecture for irreducible actions of $\SL_2(\mathbb{R}) \times \SL_2(\mathbb{R})$.

The proof relies on a new stabilizer rigidity criterion using
local spectral gap of projected stabilizers in a single simple
factor. Here local spectral gap is understood in the sense of
Boutonnet, Ioana and Salehi Golsefidy.
\end{abstract}

\section{Introduction}

In this article we prove a new criterion for a dynamical system to satisfy the conclusions of the Stuck--Zimmer theorem. Instead of relying on global gap results, we show how to reduce the proof to \emph{local spectral gap} in the sense of Boutonnet--Ioana--Salehi Golsefidy \cite{BISG} in one of the factors. Let us first state the two main consequences of this criterion.

\smallbreak

Work of Boutonnet--Ioana--Salehi Golsefidy provides conditions under which a dense subgroup has local spectral gap. Their result (\cite[Theorem A]{BISG}) shows that this criterion is satisfied under assumptions on the algebraicity of entries. Together with Margulis' arithmeticity and
Theorem~\ref{Theorem: A local spectral gap criterion} below,
we therefore obtain:

\begin{theorem}[Stuck--Zimmer for irreducible lattices]\label{Cor: Lattice case}
    Let $G$ be a connected semisimple real Lie group with finite center, no compact factors and real rank at least $2$ and let $\Gamma$ be an irreducible lattice in $G$. Let $\Gamma\curvearrowright(X,\mu)$ be an ergodic probability preserving action on a
standard probability space. Then for $\mu$ almost every $x$, the stabiliser $\Gamma_x$ at $x$ either has finite index in $\Gamma$ or is central.
\end{theorem}

When one or more factors of $G$ are $\SL_2(\mathbb{R})$, we can establish local spectral gap without assuming arithmeticity. We highlight that \emph{no} algebraicity of entries is needed here. In particular:

\begin{theorem}[Stuck--Zimmer with an $\SL_2(\mathbb R)$ factor]\label{Cor: SL2 case}
    Let $G$ be a connected semisimple real Lie group
with finite center, no compact factors and at least two simple factors. Assume that one of the factors is isogenous to $\mathrm{SL}_2(\mathbb{R})$. Let
$G\curvearrowright(X,\mu)$ be a probability preserving action on a
standard probability space. Assume that
the $G$-action is irreducible (i.e. every almost simple factor acts ergodically).

Then either the action is concentrated on a fixed point, or for almost every $x \in X$, $G_x$ is either a lattice or central.
\end{theorem}

\begin{remark}
    In particular, this establishes the Stuck--Zimmer conjecture unconditionally for $\mathrm{SL}_2(\mathbb{R}) \times \mathrm{SL}_2(\mathbb{R})$.
\end{remark}
\smallbreak

\subsection{Historical background} Margulis' normal subgroup theorem \cite{Margulis78,Margulis79}
already exhibits the difficulty caused by the absence of
property~$(T)$. For a noncentral normal subgroup $N$ of an
irreducible higher rank lattice $\Gamma$, its proof establishes
that $\Gamma/N$ is amenable and has property~$(T)$, hence is
finite. The second assertion is immediate when $\Gamma$ itself
has property~$(T)$. Margulis' additional argument for products
of rank one groups shows how the interaction between the
factors can supply the missing rigidity \cite{Margulis79}. Our proof will use a similar transfer of almost invariance between the factors, with local spectral gap giving quantitative control.

In 1994, Stuck and Zimmer \cite{SZ} proved that every faithful
irreducible probability preserving action of a connected semisimple
real Lie group with finite center, no compact factors, real rank
at least two and property $(T)$ is either essentially free or
essentially transitive. Their corresponding theorem for lattices
extends Margulis' normal subgroup theorem from normal subgroups
to stabilizers. It is widely believed that property (T) is unnecessary and that sufficient irreducibility should be enough: the \emph{Stuck--Zimmer conjecture} asks precisely whether the
property $(T)$ assumption can be removed, see \cite[p.~731]{SZ} and
\cite[Conjecture~5.3]{GelanderSurvey}.

Bader and Shalom \cite{BS} extended the intermediate factor method
to products of locally compact groups.
Creutz \cite{CreutzProducts} and, independently, Hartman and Tamuz
\cite{HT} showed that, in the Lie group setting, it suffices that
one simple factor have property $(T)$. The co-amenability argument has its origins in the intermediate
factor method of Stuck and Zimmer \cite{SZ}. We use it in the
product form proved by Hartman and Tamuz
\cite[Theorem~1]{HT}, which gives co-amenability of the stabilizers
under our hypotheses. The question is then whether these
stabilizers have finite covolume. Our criterion obtains this
last step from local spectral gap in one factor.
\smallbreak
More recently, Fr\k{a}czyk and Gelander \cite{FG} proved that
confined discrete subgroups of connected center-free simple Lie
groups of real rank at least two are lattices.
Bader, Gelander and Levit \cite{BGL} developed a spectral gap
theorem for products and obtained a version of the Stuck--Zimmer
conjecture under a stronger irreducibility assumption.
For lattice actions, Dogon, Glasner, Gorfine, Hanany and Levit
 established in the breakthrough work \cite{DGGHL}the conjecture for non-uniform irreducible
higher-rank lattices in semisimple real Lie groups through
character rigidity.

There is also a geometric reason to care about this conjecture.
Invariant random subgroups allow one to study locally symmetric
spaces from the point of view of a random basepoint.
In \cite{7samurai} for instance, this viewpoint is combined with the
Stuck--Zimmer theorem and property~$(T)$ to control limits of
these spaces and study the asymptotic behaviour of Betti
numbers and spectral invariants.
\smallbreak

\subsection{A criterion from local spectral gap in a factor}
The common ingredient in these two results is the following criterion. It relies on the notion of \emph{local spectral gap} which we recall now. countable dense subgroup $\Delta$ of a connected
simple Lie group $G$ has local spectral gap if, for some
(equivalently, every) compact neighborhood $B$ of the identity,
there are a finite set $E\subset\Delta$ and $\kappa>0$ such that
\[
 \int_B|F(b)-\beta_B(F)|^2\,d\beta_B(b)
 \leq\kappa\int_B\max_{\gamma\in E}
       |F(b\gamma)-F(b)|^2\,d\beta_B(b)
\]
for every $F\in L^2_{loc}(G)$, where $\beta_B$ denotes normalized
Haar measure on $B$.

\begin{theorem}[Rigidity of stabilizers via local spectral gap]\label{Theorem: A local spectral gap criterion}
    Let $G=G_1\times G_2$ be a product of connected semisimple real Lie groups
with finite center and no compact factors, with $G_1$ simple. Let
$G\curvearrowright(X,\mu)$ be a probability preserving action on a
standard probability space. Assume that
both the  $G_1$-action  and $G_2$-action are ergodic and that, for almost every $x$, the stabilizer
$G_x$ is discrete, $p_1(G_x)$  is dense in $G_1$ and has local spectral
gap on $G_1$ and $p_2(G_x)$ is dense in $G_2$.

Then for almost every $x \in X$, $G_x$ is  a lattice.
\end{theorem}

\smallbreak

The proof starts from the almost invariant vectors supplied by
Hartman--Tamuz's co-amenability theorem. Since stabilizer elements
fix points, almost invariance under $G_2$ gives almost invariance
under the projection of each stabilizer to $G_1$. Local spectral
gap then forces these vectors to be almost constant on compact
subsets of $G_1$-orbits, wherever the gap is uniformly controlled.
If a vector has substantial mass there, its averages over these
compact sets must therefore be large in $L^2$. On the one hand,
we use the ergodic theorem to bring a definite proportion of
the vector's mass into this set of controlled local spectral gap,
after averaging along a flow in $G_1$. On the other hand, if the
orbit has infinite volume, Howe--Moore makes the corresponding
compact averages tend to zero along the same flow, giving a
contradiction. 
\smallbreak
Theorem~\ref{Theorem: A local spectral gap criterion} also
suggests a route to the full Stuck--Zimmer conjecture for
actions of semisimple groups. Our estimate for escape from identity neighbourhoods
(Lemma~\ref{lem:height-separation}) already applies to arbitrary
simple factors. To obtain local spectral gap, we still need
quantitative escape from neighbourhoods of all proper connected
closed subgroups, as required by \cite[Theorem~6.7]{BISG}.
For $\SL_2(\mathbb R)$, we use the fact that these subgroups
are $2$-step solvable. In larger rank one groups, proper connected
subgroups can themselves be semisimple, so this argument does
not extend directly. Obtaining the corresponding escape estimates
would require new ideas and further work (as is illustrated by the analogous story of passing from product theorems for $\SL_2(\mathbb F_p)$ to finite simple groups of Lie type \cite{Helfgott,BGT,PS}) and is beyond the scope of this paper. For the conjecture, it would
be enough to establish local spectral gap in one factor
for the projected stabilizers arising in irreducible
probability preserving actions.
\subsection{Organisation of the paper}
We start in Section~\ref{sec:criterion} by proving the local spectral gap criterion (Theorem \ref{Theorem: A local spectral gap criterion}).

In Section \ref{sec:corollaries} we then establish the two main theorems. Outside the classical simple case, Theorem \ref{Cor: Lattice case} is a consequence of three facts: the usual induction lifts the $\Gamma$-system to an irreducible $G$-system when its stabilizers are not central; Margulis' arithmeticity implies that $\Gamma$ is arithmetic; and \cite{BISG} then gives local spectral gap out of the box.

For Theorem \ref{Cor: SL2 case}, we do not want to rely on algebraicity of coefficients, as it is not known \emph{a priori} in this case. We rely instead on a novel counting argument for discrete subgroups in products, which gives the \emph{escape from subgroups} estimates needed to apply \cite{BISG} to their projections to a factor.

\subsection{Acknowledgement} The authors are grateful to Emmanuel Breuillard, Mikolaj Fraczyk and Arie Levit for valuable advice. 
SM acknowledges support from ETH Zurich, and OpenAI through the``OpenAI for academic research'' scheme. YY was supported by SNF grant 200020--212617.

GPT6 Astra was used to help with the literature review, explore research leads, carry out some computations and proofread the manuscript. The main novel ideas of the paper are the authors' own  (and mostly predate modern LLMs) and computations carried out by LLMs were, in the end, not used. 

\section{The local spectral gap criterion}\label{sec:criterion}

In this section we prove the local spectral gap criterion (Theorem \ref{Theorem: A local spectral gap criterion}). Thus, throughout this section let $G=G_1\times G_2$ be a product of connected semisimple real Lie groups
with finite center and no compact factors, with $G_1$ simple. Let
$G\curvearrowright(X,\mu)$ be a probability preserving action on a
standard probability space. Assume that
both the  $G_1$-action  and $G_2$-action are ergodic and that, for almost every $x$, the stabilizer
$G_x$ is discrete, $p_1(G_x)$  is dense in $G_1$ and has local spectral
gap on $G_1$. Assume also that $p_2(G_x)$ is dense in $G_2$.

Recall that a countable dense subgroup $\Delta$ of a connected
simple Lie group $G_1$ has local spectral gap if, for some
(equivalently, every) compact neighborhood $B$ of the identity,
there are a finite set $E\subset\Delta$ and $\kappa>0$ such that
\[
 \int_B|F(b)-\beta_B(F)|^2\,d\beta_B(b)
 \leq\kappa\int_B\max_{\gamma\in E}
       |F(b\gamma)-F(b)|^2\,d\beta_B(b)
\]
for every $F\in L^2_{loc}(G_1)$, where $\beta_B$ denotes normalized
Haar measure on $B$. We use here for convenience the right translation
form of \cite[Propositions~2.2--2.3 and Remark~1.5]{BISG}.

Fix once and for all a regular split one parameter subgroup $(a_t)_{t\in\mathbb R}$
of $G_1$. Its action on $(X,\mu)$ is ergodic by Howe--Moore
\cite[Theorem~1.1]{Ciobotaru}. We will choose one orbit
$Y=Gx\simeq G/G_x$ and work in $L^2(Y,m_Y)$, where $Y$ carries the
quotient topology and $m_Y$ is quotient Haar measure. Write $\pi(g)u(y):=u(g^{-1}y)$.
\smallbreak

For a compact $C\subset G_1$ of positive Haar measure, let $\beta_C$
be normalized Haar measure on $C$ and set the Markov operator
\[
 M_Cu(y):=\int_Cu(cy)\,d\beta_C(c).
\]
For $u\in C(Y)\cap L^2(Y)$ and a compact $Q\subset G_2$, write
\[
 D_Q(u):=\left\|\sup_{\ell\in Q}
             |u(\ell^{-1}\,\cdot)-u|\right\|_{L^2(Y)}.
\]
\subsection{From local spectral gap to an orbit estimate}

The first lemma exploits the following elementary observation: for every
$y\in Y $ where $Y = Gx$ for some $x \in X$, and $\gamma=(\gamma_1,\gamma_2)\in G_y$ we have
$\gamma_1y=\gamma_2^{-1}y$. Since the two factors commute, for
$u\in C(Y)\cap L^2(Y)$ we can therefore write
\[
 u(b\gamma_1y)-u(by)
 =u(\gamma_2^{-1}by)-u(by), \qquad b\in G_1.
\]
The left-hand side is the kind of difference that appears in the local
spectral gap inequality on $G_1$, while the right-hand side involves
translations in $G_2$. So the idea is to use local spectral gap to control
local variance along $G_1$ by the defect $D_Q(u)$.\footnote{This transfer is reminiscent of Margulis' argument
for products \cite[Theorem~1.3.2]{Margulis79}. There, the
representation being induced is trivial on the normal subgroup,
and smoothing together with density propagates almost
invariance. Here the relation comes from the stabilizer at
each point, and local spectral gap gives a quantitative
estimate for the transfer.}

To make this work, we still need to choose a set of positive measure on
which the gap constants and the required $G_2$-components are uniformly
controlled, and then integrate on the orbit. The next lemma shows
precisely what estimate this gives: a bound for the weighted $L^2$-mass
of $u$ in terms of its compact average $M_Cu$ and its defect $D_Q(u)$.

\begin{lemma}[Local gap and compact averages]\label{lem:local-transfer}
For every compact $C\subset G_1$ of positive Haar measure, there are a
measurable set $X_0\subset X$ with $\mu(X_0)>1/2$, a compact
$Q\subset G_2$, and $\kappa>0$ with the following property. Put
\[
 w(y):=\int_C\mathbf1_{X_0}(c^{-1}y)\,d\beta_C(c).
\]
On every orbit $Y=Gx$ with discrete stabilizer, every $u\in C(Y)\cap L^2(Y)$
with $D_Q(u)<\infty$ satisfies
\begin{equation}\label{eq:local-transfer}
 \|M_Cu\|_2^2+\kappa D_Q(u)^2
 \geq\int_Y w(y)|u(y)|^2\,dm_Y(y).
\end{equation}
Moreover, $0\leq w\leq1$ and $\int_Xw\,d\mu=\mu(X_0)>1/2$.
\end{lemma}

\begin{proof}
Choose a compact neighborhood $B\supset C$ in $G_1$. Note that local spectral gap holds for any such $B$ if it holds for one
\cite[Proposition~2.3 and Remark~1.5]{BISG}. We use it in its right
translation form.
Now for almost every $y\in X$, local spectral gap gives a finite
set $E_y\subset G_y$ and an integer gap constant $\kappa(y)$ such that 
\[
 \int_B|F(b)-\beta_B(F)|^2\,d\beta_B(b)
 \leq\kappa(y)\int_B\max_{\gamma\in E_y}
       |F(b\gamma_1)-F(b)|^2\,d\beta_B(b)
\]
for every $F\in L^2_{loc}(G_1)$ and where $\gamma_1$ denotes the $G_1$-component of $\gamma$. The data $\kappa(y)$ and $E_y$ may be chosen
measurably in $y$, for instance by enumerating the discrete stabilizers and testing the
inequality on a countable dense subset of $L^2_{loc}(G_1)$.
Using a compact exhaustion of $G_2$, choose a measurable set
$X_0\subset X$ with $\mu(X_0)>1/2$ such that $\kappa(y)$,
$|E_y|$, and the $G_2$-components of elements from $E_y$
are all bounded bounded. Thus there are a compact $Q\subset G_2$ and $\kappa>0$
such that, for every $y\in X_0$, we have $p_2(E_y)\subset Q$ and
\begin{equation}\label{eq:uniform-local-gap}
 \int_C|F(c) - \beta_C(F) |^2d\beta_C(c)\leq\kappa\int_B
       \max_{\gamma\in E_y}|F(b\gamma_1)-F(b)|^2\,d\beta_B(b) 
\end{equation}
for every $F \in L^2_{loc}(G_1)$.

\smallbreak

Apply
\eqref{eq:uniform-local-gap} to $F_y(b):=u(by)$ for $y\in X_0\cap Y$.
For $\gamma=(\gamma_1,\gamma_2)\in G_y$,
\begin{equation}\label{eq:stabilizer-identity}
 F_y(b\gamma_1)-F_y(b) = u(b\gamma_1y)-u(by)
   =u(\gamma_2^{-1}by)-u(by)
\end{equation}
since $b\in G_1$ commutes with $\gamma_2\in G_2$. Note that the orbit functions $F_y$ are locally square integrable for almost every
$y$ by Fubini. Integrating the variance inequality \eqref{eq:uniform-local-gap} over $X_0\cap Y$ gives 

\begin{align}
\int_{X_0 \cap Y}\int_C|F_y(c) - \beta_C( F_y ) |^2d\beta_C(c)dm_Y(y) &\leq  \int_{X_0 \cap Y}\kappa\int_B
       \max_{\gamma\in E_y}|F_y(b\gamma_1)-F_y(b)|^2\,d\beta_B(b) dm_Y(y)
\\
& = \int_B\kappa\int_{X_0 \cap Y} \max_{\gamma \in E_y} |u(\gamma_2^{-1}by) - u(by)|^2dm_Y(y) d\beta_B(b) \\
& \leq \kappa D_Q(u)^2. \label{Control by D_Q}
\end{align}

Now, by definition of $w$,
\[
 \int_Yw(y)|u(y)|^2\,dm_Y(y)
 =\int_{X_0\cap Y}\int_C|u(cy)|^2\,d\beta_C(c)\,dm_Y(y).
\]

Since $\beta_C(F_y)=(M_Cu)(y)$, adding back the squared mean gives 

\[
 \int_Y w(y)|u(y)|^2\,dm_Y(y)
 \leq\|M_Cu\|_2^2+\kappa D_Q(u)^2.
\]

The assertion about $\int_Xw\,d\mu$ is a direct consequence of invariance of $\mu$.
\end{proof}

\begin{remark}
The weight $w$ is the indicator of $X_0$ averaged over $C$. Thus, we can
control the $L^2$ mass of $u$ seen by this weight through the $L^2$ mass
of its compact average $M_Cu$ and the defect $D_Q(u)$.

When $u$ is almost invariant in the $G_2$-direction, the defect $D_Q(u)$
should be small after averaging along $G_2$. So, if a definite proportion
of the mass of $u$ is seen by $w$, its compact average $M_Cu$ must also
carry a definite amount of mass. We now put this intuition to work.
\end{remark}

The right side of \eqref{eq:local-transfer} is the mass seen from
$X_0$ after averaging over $C$. To use the estimate for an arbitrary
vector, we return a fixed proportion of its mass to this weight.
\smallbreak

\begin{lemma}[Mass on a generic orbit]\label{lem:local-input}
Fix $C$, $X_0$, $Q$, $\kappa$, and $w$ as in
Lemma~\ref{lem:local-transfer}. For $\mu$-almost every $x$, the orbit
$Y:=Gx$ has the following property for every $u\in L^2(Y)$:
\begin{equation}\label{eq:mass-return}
 \frac1T\int_0^T\int_Yw(y)|\pi(a_t)u(y)|^2\,dm_Y(y)\,dt
 \longrightarrow\mu(X_0)\|u\|_2^2.
\end{equation}
Consequently, if $u\in C(Y)\cap L^2(Y)$ and $D_Q(u)<\infty$, there are arbitrarily large
$t\geq0$ for which
\begin{equation}\label{eq:local-lower-bound}
 \|M_C\pi(a_t)u\|_2^2
 \geq\frac12\|u\|_2^2-\kappa D_Q(u)^2.
\end{equation}
The orbit can be chosen before $u$.

\end{lemma}

\begin{proof}
The flow $a_t$ is ergodic on $(X,\mu)$. By the pointwise ergodic
theorem,
\begin{equation}\label{eq:generic-flow}
 \frac1T\int_0^T w(a_ty)\,dt\longrightarrow\mu(X_0)
\end{equation}
for $\mu$-almost every $y$. Let $N$ be the exceptional null set.
For every compact $D\subset G$, invariance of $\mu$ gives
\[
 \int_X\int_D\mathbf1_N(gx)\,dg\,d\mu(x)
 =m_G(D)\mu(N)=0.
\]
Fubini's theorem and a countable compact exhaustion of $G$ show that,
for $\mu$-almost every $x$, equation~\eqref{eq:generic-flow} holds at
$m_{Gx}$-almost every point of $Gx$. Fix one such orbit $Y=Gx$.
This choice is independent of $u$.
\smallbreak

For every $u\in L^2(Y)$, changing variables gives
\[
 \int_Yw(y)|\pi(a_t)u(y)|^2\,dm_Y(y)
 =\int_Yw(a_ty)|u(y)|^2\,dm_Y(y).
\]
Equation~\eqref{eq:generic-flow} and dominated convergence now prove
\eqref{eq:mass-return}.
For the consequence, assume that $u$ is continuous and $D_Q(u)<\infty$.
If $u\neq0$, this limit is strictly greater than
$\frac12\|u\|_2^2$. Hence the inner integral is at least
$\frac12\|u\|_2^2$ for arbitrarily large $t$.
The commuting factors give
\[
 D_Q(\pi(a_t)u)=D_Q(u).
\]
Apply \eqref{eq:local-transfer} at these times to obtain
\eqref{eq:local-lower-bound}. The case $u=0$ is immediate.
\end{proof}

We now compare this return of mass with decay of matrix coefficients
on the orbit.
\smallbreak

\begin{corollary}[Displacement on a generic orbit]\label{cor:displacement}
Assume also that $p_2(G_x)$ is dense in $G_2$ almost surely. For
$\mu$-almost every $x$, put $\Gamma:=G_x$ and $Y:=G/\Gamma$.
With $Q$ and $\kappa$ as in Lemma~\ref{lem:local-input},
\begin{equation}\label{eq:displacement-gap}
 D_Q(u)\geq\frac{1}{\sqrt{2\kappa}}\|u\|_2
 \qquad\text{for every }u\in C(Y)\cap\bigl(L^2(Y)\ominus L^2(Y)^{G_1}\bigr).
\end{equation}
If $\Gamma$ has infinite covolume, this holds for every $u\in C(Y)\cap L^2(Y)$.
If $\Gamma$ is a lattice, it holds for every $u\in C(Y)\cap L^2_0(Y)$.
\end{corollary}

\begin{proof}
A $G_1$-invariant function on $G/\Gamma$ lifts to a function $v$ on
$G_2$, with $v(\ell\gamma_2)=v(\ell)$ for every $\gamma\in\Gamma$.
Density of $p_2(\Gamma)$ and continuity of translations in
$L^2_{loc}(G_2)$ imply that $v$ is constant. Thus $L^2(Y)^{G_1}$ is zero
when $\Gamma$ has infinite covolume and consists of the constants otherwise.
\smallbreak

Now we claim that Howe--Moore implies,
\begin{equation}\label{eq:averaging-decay}
 \|M_C\pi(a_t)u\|_2\longrightarrow0.
\end{equation}
Indeed, writing $\check\beta_C$ for the image of $\beta_C$ under
the inverse map,
\[
 \|M_C\pi(a_t)u\|_2^2
 =\int_{G_1}\langle\pi(a_t^{-1}ba_t)u,u\rangle\,
                         d(\beta_C*\check\beta_C)(b).
\]
The conjugates $a_t^{-1}ba_t$ tend to infinity outside a proper
parabolic subgroup, which has Haar measure zero. Since $\beta_C*\check\beta_C$ is absolutely continuous, Howe--Moore's matrix coefficient decay gives $\langle\pi(a_t^{-1}ba_t)u,u\rangle \underset{t \rightarrow\infty}{ \longrightarrow } 0$ for $\beta_C*\check\beta_C$-almost all $b$. Then dominated convergence yields \eqref{eq:averaging-decay}.
\smallbreak

Let $t\to\infty$ along the times supplied by
Lemma~\ref{lem:local-input}. Equation~\eqref{eq:local-lower-bound}
gives
\[
 \kappa D_Q(u)^2\geq\frac12\|u\|_2^2,
\]
which is \eqref{eq:displacement-gap}.
\end{proof}

\subsection{Co-amenability and concluding the proof} To conclude, we use \cite[Theorem~1]{HT}, whose product
formulation only asks that $G_1$ and $G_2$ act ergodically,
even when $G_2$ has several simple factors. It supplies the
almost invariant vectors needed for the final contradiction.

\begin{proof}[Proof of Theorem \ref{Theorem: A local spectral gap criterion}.]
Fix a compact neighborhood $C\subset G_1$ and choose an orbit
$Y=G/\Gamma$ to which Lemma~\ref{lem:local-input} applies.
By a result of Hartman--Tamuz \cite[Theorem~1]{HT}, for almost every
such orbit $\Gamma$ is co-amenable in a product $H_1 \times H_2$ of closed normal
subgroups of $H_1 \subset G_1$ and $H_2 \subset G_2$. Since $H_1$ contains the closure of the projection of $\Gamma$ to $H_1$, and $\Gamma$ projects densely to $G_1$, $H_1 = G_1$. Similarly, $H_2 = G_2$. So $H_1 \times H_2 = G_1 \times G_2$ and $\Gamma$ is co-amenable in $G_1 \times G_2$. In other words, there are almost $G$-invariant unit
vectors $f_n\in L^2(Y)$. By density, we may take $f_n\in C_c(Y)$.
\smallbreak

Suppose $\Gamma$ has infinite covolume. Choose a nonnegative
$\psi\in C_c(G_2)$ of integral one and set
$u_n:=\psi *f_n\in C_c(Y)$. Then $\|u_n\|_2\to1$. Let
$Q$ be the compact subset from Lemma~\ref{lem:local-input}, and put
$R:=Q\supp(\psi)\cup\supp(\psi)$. For $\ell\in Q$ and $\bar g\in Y$,
\begin{align*}
    (\pi(\ell)u_n-u_n)(\bar g)
 & =\int_{G_2}\psi(h)
       [f_n(h^{-1}\ell^{-1}\bar g)-f_n(h^{-1}\bar g)]\,dh \\
       & = \int_{\supp(\psi)}\psi(h)
       [f_n(h^{-1}\ell^{-1}\bar g)-f_n(h^{-1}\bar g)]\,dh \\ 
       & = \int_{\supp(\psi)}\psi(h)
       [f_n(h^{-1}\ell^{-1}\bar g)-f_n(\bar g)]\,dh - \int_{\supp(\psi)}\psi(h)
       [f_n(h^{-1}\bar g)-f_n(\bar g)]\,dh\\ 
       & = \int_{R}[\psi(l^{-1}h) - \psi(h)]
       [f_n(h^{-1}\bar g)-f_n(\bar g)]\,dh
\end{align*}
where the last line is obtained by a a change of variable $lh \rightarrow h$ in the left-hand integral.
Hence
\[
 \sup_{\ell\in Q}|\pi(\ell)u_n(\bar g)-u_n(\bar g)|
 \leq2\|\psi\|_\infty\int_R|f_n(h^{-1}\bar g)-f_n(\bar g)|\,dh.
\]
Taking the $L^2$-norm and using Minkowski's  inequality gives
\begin{equation}\label{eq:smoothed-errors}
 D_Q(u_n)\leq2\|\psi\|_\infty 
          \int_{R}\|\pi(h)f_n-f_n\|_2\,dh\longrightarrow0.
\end{equation}
\smallbreak

Corollary~\ref{cor:displacement} now gives
\[
 D_Q(u_n)\geq\frac{1}{\sqrt{2\kappa}}\|u_n\|_2.
\]
This contradicts \eqref{eq:smoothed-errors}, since $\|u_n\|_2\to1$.
\end{proof}

\section{Proofs of the main theorems}\label{sec:corollaries}

We deduce here the two main theorems from the local spectral gap criterion.

\subsection{The lattice case}
The following is a direct consequence of
\cite[Corollary~6.10]{BGL} after passing to the quotient by the finite
center.

\begin{lemma}[Irreducibility of the induced action]\label{lem:induction-irreducibility}
Let $G$ be a connected semisimple real Lie group with finite center,
no compact factors, and at least two simple factors. Let $\Gamma<G$
be an irreducible lattice, and let $\Gamma\curvearrowright(X,\mu)$
be an ergodic probability preserving action on a standard probability
space. Then either $\Gamma_x\subseteq Z(G)$ for almost every $x$, or
the induced action $G\curvearrowright G\times_\Gamma X$ is
irreducible.
\end{lemma}

\begin{proof}
Put $Z:=G\times_\Gamma X$. First suppose that $G$ has trivial
center and that $\Gamma\curvearrowright X$ is not essentially free.
The induced action is ergodic, and its stabilizers are
\[
 G_{[g,x]}=g\Gamma_xg^{-1}.
\]
They are discrete and, by ergodicity, nontrivial almost everywhere.
Irreducibility of $\Gamma$ gives $\Gamma\cap N=\{e\}$ for every
proper semisimple factor $N\lhd G$. Thus
\[
 G_{[g,x]}\cap N=g(\Gamma_x\cap N)g^{-1}=\{e\}.
\]
The stabilizer distribution is therefore a nontrivial discrete
ergodic invariant random subgroup satisfying the hypotheses of
\cite[Corollary~6.10]{BGL}. The last assertion of that corollary
shows that $G\curvearrowright Z$ is irreducible.
\smallbreak

For finite center, put $C:=\Gamma\cap Z(G)$,
$\overline G:=G/Z(G)$, $\overline\Gamma:=\Gamma/C$, and
$\overline X:=X/C$. The induced $\overline G$-space is naturally
$\overline Z:=Z/Z(G)$. If $\overline\Gamma\curvearrowright
\overline X$ is essentially free, then $\Gamma_x\subseteq C$
almost surely. Otherwise the case just proved shows that
$\overline G\curvearrowright\overline Z$ is irreducible.
For any simple factor $S$ of $G$, the space of $S$-ergodic
components of $Z$ has trivial quotient by the finite group $Z(G)$,
so it is finite. Since $G$ is connected and acts ergodically on $Z$,
this finite factor is a point. Thus every simple factor acts
ergodically on $Z$.
\end{proof}

We use the standard invariant-map argument for projections of
stabilizers. Compare the proof of \cite[Corollary~1.5]{HT}.

\begin{lemma}[Density of projected stabilizers]\label{lem:dense-projections}
Let $G$ be a product of at least two connected noncompact simple real
Lie groups with finite center. Let $G\curvearrowright(Z,\nu)$ be an
irreducible probability preserving action with discrete stabilizers.
Then either almost every stabilizer is central, or almost every
stabilizer projects densely to every nontrivial proper subproduct of
$G$.
\end{lemma}

\begin{proof}
Write $G=H\times K$, where $H$ is a nontrivial proper subproduct.
The map $z\mapsto\overline{p_H(G_z)}$ is $K$-invariant. By
irreducibility it is almost surely constant, and equivariance shows
that its value $N$ is a closed normal subgroup of $H$.
If $N\neq H$, its projection to some simple factor $S$ of $H$ is
central. Conjugation by $S$ therefore fixes $G_z$ almost surely.
Since $S$ acts ergodically, the stabilizer map itself is almost surely
constant. Its value is a discrete normal subgroup of the connected
group $G$, hence is central. Outside this case, $N=H$ for every
nontrivial proper subproduct $H$.
\end{proof}

\begin{proof}[Proof of Theorem \ref{Cor: Lattice case}.]
Put $C:=\Gamma\cap Z(G)$. The stabilizer of the image $\overline x$
in $X/C$ for the $\Gamma/C$-action is $\Gamma_xC/C$. Since $C$ is
finite, it suffices to prove the assertion after passing to $G/Z(G)$.
We may therefore assume that $G$ has trivial center.
If $G$ is simple, the conclusion is classical \cite{SZ}. We may now
assume that $G$ has at least two simple factors.
\smallbreak

    Let $(X,\mu)$ be the ergodic $\Gamma$ system. Let
$Z:=G\times_\Gamma X$ be the induced $G$ system, where
$\gamma\cdot(g,x):=(g\gamma^{-1},\gamma x)$. Equip it with the
induced probability measure $\nu$, obtained from normalized Haar
measure on $G/\Gamma$ and $\mu$ on the fibers
\cite[Section~3.1]{Creutz}.
\smallbreak

For $z=[g,x]$, the stabilizer is
\[
 G_z=g\Gamma_xg^{-1}\subseteq g\Gamma g^{-1}.
\]
By Margulis' arithmeticity theorem, $\Gamma$ is arithmetic.
Thus $G_z$ is contained in a conjugate of an arithmetic lattice, so its
adjoint matrices have algebraic entries in a suitable basis.
By Lemma \ref{lem:induction-irreducibility}, either $\Gamma_x$ is central
almost surely, in which case there is nothing to prove, or the induced
$G$-action is irreducible. In the latter case,
Lemma~\ref{lem:dense-projections} gives dense projections to every proper subproduct.
\cite[Theorem~A]{BISG} gives local spectral gap for the projected
stabilizers, and Theorem~\ref{Theorem: A local spectral gap criterion}
implies that $G_z$ is a lattice. Since $\Gamma_x\leq\Gamma$, this
means $[\Gamma:\Gamma_x]<\infty$.
\end{proof}

\subsection{The case of an
\texorpdfstring{$\SL_2(\mathbb R)$}{SL2(R)} factor}\label{sec:sl2}

In this case, discreteness in the product supplies the estimate for
escape from identity neighbourhoods usually obtained from algebraicity.
A projected element can be very close to the identity only if its
complementary component is large. For nontrivial words in a fixed finite
set of generators, this means that they stay outside identity
neighbourhoods whose radii shrink exponentially with word length.
The special feature of $\mathrm{PSL}_2(\mathbb R)$ is that its proper
connected subgroups satisfy a common group law. Together, these two
facts give the escape from subgroups required by
\cite[Theorem~6.7]{BISG}.
\smallbreak

\begin{proposition}[Local spectral gap for the rank-one projection]
\label{prop:sl2-projection}
Let $G_2$ be a nontrivial connected semisimple real Lie group with finite
center and no compact factors, and let
$G_1=\mathrm{SL}_2(\mathbb R)$ or $\mathrm{PSL}_2(\mathbb R)$.
If $\Gamma<G_1\times G_2$ is discrete and projects densely to both
factors, then $p_1(\Gamma)\curvearrowright G_1$ has local spectral gap.
\end{proposition}

The argument for escape from identity neighbourhoods works for any simple projected factor.
For a semisimple group $M$, fix a Euclidean norm on its Lie algebra
and put
\[
 N_M(g):=\max\{
 \|\operatorname{Ad}(g)\|_{\mathrm{op}},
 \|\operatorname{Ad}(g^{-1})\|_{\mathrm{op}}\}.
\]
We use the Hilbert--Schmidt norm for matrix differences. 

\begin{lemma}[Quantitative escape from identity neighbourhoods]
\label{lem:height-separation}
Let $G_2$ be as in Proposition~\ref{prop:sl2-projection}, let $G_1$ be a
connected noncompact simple real Lie group with finite center, and
let $\Gamma<G_1\times G_2$ be discrete with dense projections.
There are $c,B>0$ such that
\begin{equation}\label{eq:sl2-height}
 \forall \gamma\in\Gamma,\ \gamma_1\notin Z(G_1), \qquad \|\operatorname{Ad}(\gamma_1)-I\|
 \ge cN_{G_2}(\gamma_2)^{-B}.
\end{equation}
\end{lemma}

\begin{proof}
Identify as we may each factor with its adjoint group. In particular, we view $G_1 \times G_2$ as a subgroup of the group of endomorphisms of its Lie algebra.
Equip each factor $M$ with the right-invariant Riemannian distance
$d_M$ induced by the chosen norm on its Lie algebra. Near the
identity, $d_M(g,e)$ is comparable to
$\|\operatorname{Ad}(g)-I\|$. It therefore suffices to prove
\eqref{eq:sl2-height} with $d_{G_1}(\gamma_1,e)$ on the left.
\smallbreak

We shall first construct a ``mesh'' of elements of $\Gamma$ at every small scale, with small $G_2$-components and controlled $G_1$-heights. We will then look at how elements of $\Gamma$ act on this mesh by conjugation. An element violating \eqref{eq:sl2-height} will move every point of the mesh by only a very small amount. Discreteness of $\Gamma$ will force it to fix the mesh pointwise, and the way we have chosen the mesh will then give a contradiction.

Let us build this mesh now. For every sufficiently small $s>0$, we will find a finite
set $\mathcal M_s\subset\Gamma$ of uniformly bounded cardinality
such that for all $\xi := (\xi_1,\xi_2)  \in \mathcal{M}_s$
\begin{equation}\label{eq:sl2-mesh}
 \begin{gathered}
 d_{G_2}(\xi_2,e)\le s,\qquad
 N_{G_1}(\xi_1)\le C s^{-b}
 \end{gathered}
\end{equation}
and 
\[\bigcap_{\xi\in\mathcal M_s}C_{G_1}(\xi_1)=\{e\}.\]
Here $b,C>0$ are fixed, and $C_{G_1}(h)$ denotes the
centralizer of $h$ in $G_1$. The last condition says that every
nontrivial conjugation in $G_1$ moves at least one point of the mesh.
\smallbreak

First we find one pivot at each scale. By
\cite[Theorem~1.1 and its proof in Section~5]{BG}, choose a free pair
$u_2,v_2\in p_2(\Gamma)$ as close to $I$ as needed, and choose lifts
$u,v\in\Gamma$. Suppose
$N_{G_2}(u_2),N_{G_2}(v_2)\le e^\eta$ and $N_{G_1}(u_1),N_{G_1}(v_1)\le e^\tau$ for some $\eta, \tau > 0$. The
$4\cdot3^{m-1}$ reduced words of length $m$ in the letters $u,v$ lie in a
the set of all matrices with entries at most $O(e^{\eta m})$. If the matrices have size $d \times d$,
partition this set into $d^2$-dimensional cubes of side
$c_0e^{-(\alpha+\eta)m}$. Choose $\alpha,\eta>0$ sufficiently small so that
\[
 d^2(\alpha+2\eta)<\log3.
\]
For all large $m$, two distinct words $w_{1,m},w_{2,m}$ in the letters $u,v$ lie in the same cube. Their quotient
gives $w_{1,m}w_{2,m}^{-1} =:\xi_m\in\Gamma$ with
\begin{equation}\label{eq:sl2-auxiliary}
 0<d_{G_2}(\xi_{m,2},e)\le e^{-\alpha m},
 \qquad N_{G_1}(\xi_{m,1})\le e^{\tau m},
\end{equation}
where $\tau>0$.  Since the projection of $\Gamma$ to $G_2$ is injective.
\smallbreak

We now spread each pivot by the same finite set of conjugations.
Choose a finite symmetric set $D\subset p_1(\Gamma)$ generating
a dense subgroup, again using \cite{BG} and let $h$ be any element in $G_1 \setminus \{e\}$ seen as an element of $\mathrm{End}_{\mathbb{R}}(\mathfrak{g}_1)$. If $q=\dim G_1$, the spaces
\[
 V_j(h):=\operatorname{span}_{\mathbb R}
 \{whw^{-1}:w\in W_{\le j}(D)\}
 \subset\operatorname{End}_{\mathbb R}(\mathfrak{g}_1)
\]
where $W_{\le j}(D)$ is the set of all words in at most $j$ letters of $D$, stabilize by $j=q^2$. Once two consecutive spaces agree, the space
is $D$-invariant and hence $G_1$-invariant. Thus an element commuting
with every displayed conjugate for $j=q^2$ commutes with every
$G_1$-conjugate of $h$. For $h\ne e$, these conjugates generate a
dense normal subgroup of $G_1$. Since $G_1$ is center-free and simple,
their common centralizer is trivial. Lifting the finite word ball $W_{\le q^2}(D)$
gives a finite set $T\subset\Gamma$ such that for all $h \in G_1\setminus \{e\}$,
\begin{equation}\label{eq:sl2-conjugators}
 \bigcap_{t\in T}C_{G_1}(t_1ht_1^{-1})=\{e\}.
\end{equation}
Fixed conjugations change the bounds in
\eqref{eq:sl2-auxiliary} by only a constant factor. Thus all the
$G_2$-components of $\{t\xi_m t^{-1}:t\in T\}$ lie within
$C_1e^{-\alpha m}$ of $e$. Given  $s>0$ small, choose $m$ so that $C_1e^{-\alpha m}\le s<C_1e^{-\alpha(m-1)}$, and define
\[
 \mathcal M_s:=\{t\xi_m t^{-1}:t\in T\}.
\]
The $G_1$-heights of elements in $\mathcal{M}_s$ are bounded by $Cs^{-b}$, with $b=\tau/\alpha$ according to \eqref{eq:sl2-auxiliary}.
Together with \eqref{eq:sl2-conjugators}, the sets $\mathcal{M}_s$ are now the mesh we were looking for.
\smallbreak

Now conjugation action on the mesh together with the discreteness of $\Gamma$ will conclude.
Recall that for a right-invariant Riemannian distance $d$ on a Lie group,
conjugation by $g$ is $\|\operatorname{Ad}(g)\|_{\mathrm{op}}$-Lipschitz.
The triangle inequality consequently gives
\begin{equation}\label{eq:sl2-conjugation-displacement}
 \begin{aligned}
 d(ghg^{-1},h)
 &\le(\|\operatorname{Ad}(g)\|_{\mathrm{op}}+1)d(h,e),\\
 d(ghg^{-1},h)
 &\le(\|\operatorname{Ad}(h)\|_{\mathrm{op}}+1)d(g,e).
 \end{aligned}
\end{equation}
For the first bound, compare both points with $e$. For the second,
right invariance gives
\[
 d(ghg^{-1},h)=d(g,hgh^{-1}))
 \le d(g,e)+d(hgh^{-1},e).
\]

\smallbreak

 We will use the following elementary consequence of the discreteness of $\Gamma$ and right invariance of $d$: there is $r > 0$ such that if $\gamma, \xi \in \Gamma$ satisfy $d_{G_1}(\xi_1,\gamma_1) < r$ and $d_{G_2}(\xi_2, \gamma_2) < r$, then $\xi = \gamma$. Indeed, then $\xi\gamma^{-1}$ is in $\gamma$ and within distance at most $2r$ of the identity. For $r$ small enough it implies $\xi = \gamma$. 
 
 Fix now $\gamma\in\Gamma$ with $\gamma_1\ne e$,
and put $R=N_{G_2}(\gamma_2)$ and $\varepsilon=d_{G_1}(\gamma_1,e)$.
Applying the two bounds in
\eqref{eq:sl2-conjugation-displacement} to the respective factors
gives, for every $\xi=(\xi_1,\xi_2)\in\mathcal M_s$,
\begin{align}
 d_{G_2}(\gamma_2\xi_2\gamma_2^{-1},\xi_2)
     &\le (R+1)s,\label{eq:sl2-commutator-2}\\
 d_{G_1}(\gamma_1\xi_1\gamma_1^{-1},\xi_1)
     &\le (Cs^{-b}+1)\varepsilon.
       \label{eq:sl2-commutator-1}
\end{align}
Take $s=a/R$, where $a>0$ is fixed and sufficiently small.
Since $R\ge1$, the first bound is at most $2a < r$.
If $\varepsilon<cR^{-b}$, with $c$ sufficiently small, the second
is also smaller than $r$. So
$\gamma\xi\gamma^{-1}=\xi$ for every $\xi\in\mathcal M_s$.
The common-centralizer condition in \eqref{eq:sl2-mesh} forces
$\gamma_1=e$, a contradiction. This proves the claimed bound with
$B=b$. Local comparison of the metric with the matrix
norm, followed by decreasing $c$, gives \eqref{eq:sl2-height}.
\end{proof}

In particular, if $S$ is a fixed finite subset of $p_1(\Gamma)$,
choose lifts of its elements to $\Gamma$. Their complementary
heights grow at most exponentially with word length. Hence there
is $D_S>0$ such that
\begin{equation}\label{eq:sl2-word-separation}
  \forall w\in W_{\le n}(S) \setminus Z(G_1)), \qquad \|\operatorname{Ad}(w)-I\|\ge e^{-D_S n}
.
\end{equation}
Here $W_{\le n}(S)$ denotes the ball of radius $n$ for the generators
$S\cup S^{-1}$.
\smallbreak

We next explain how escape from identity neighbourhoods helps establish escape from subgroups.
On $\overline{G_1}=\mathrm{PSL}_2(\mathbb R)$, use the distance
$d(g,h)=\|\operatorname{Ad}(g)-\operatorname{Ad}(h)\|$ and write
$H^{(\delta)}$ for the $\delta$-neighborhood of $H$.

\begin{lemma}[Words near a proper subgroup]\label{lem:sl2-cyclic}
Let $S=\{a,b\}$ freely generate a subgroup $F<\overline{G_1}$ and
satisfy \eqref{eq:sl2-word-separation}. There is $C>0$ such that,
for every $n\ge1$ and every proper connected closed subgroup
$H<\overline{G_1}$, the set
\[
 W_{\le n}(S)\cap H^{(e^{-Cn})}
\]
is contained in a cyclic subgroup of $F$.
\end{lemma}

\begin{proof}
Every proper connected subgroup of $\mathrm{PSL}_2(\mathbb R)$ is
abelian or conjugate into the triangular group. In particular it is
$2$-step solvable and satisfies the fixed law
\[
 P(x,y):=[[x,y],x[x,y]x^{-1}]=e.
\]
This is a nontrivial free word of length at most twenty.
\smallbreak

Matrices and inverse matrices of words in $W_{\le n}(S)$ have norms
bounded exponentially in $n$. If $x,y$ in this ball are within
$\delta$ of $H$, approximate them by elements of $H$ and telescope
the products in $P$. We get
\[
\|\operatorname{Ad}(P(x,y))-I\|\le e^{C_0n}\delta
\]
for $\delta\le e^{-C_0n}$, after enlarging $C_0$.
The constant is independent of $H$.
Choose $C>C_0+20D_S$. With $\delta=e^{-Cn}$,
\eqref{eq:sl2-word-separation} forces $P(x,y)=e$.
Two noncommuting elements of a free group freely generate, so this
nontrivial word cannot vanish on them. Thus all elements in the
displayed intersection commute. The centralizer of a nontrivial
element of a free group is cyclic, which proves the claim.
\end{proof}

We now use this control near proper subgroups to obtain the escape estimates needed to apply \cite{BISG}.

\begin{proof}[Proof of Proposition~\ref{prop:sl2-projection}]
First take $G_1=\mathrm{SL}_2(\mathbb R)$, put $\Delta=p_1(\Gamma)$,
and write $\overline{G_1}=\mathrm{PSL}_2(\mathbb R)$.
Choose a free pair $S\subset\Delta$ sufficiently close to the
identity, using \cite[Theorem~1.1 and its proof]{BG}.
Its free group embeds in $\overline{G_1}$.
Lemma~\ref{lem:height-separation} gives escape from exponentially
small identity neighbourhoods for nontrivial words, so
Lemma~\ref{lem:sl2-cyclic} applies with a fixed constant $C>0$.
\smallbreak

We use the packing construction in
\cite[proof of Theorem~3.1, Section~3.3, Part~1]{BISG}.
Applied to $S$ in $\mathrm{SL}_2(\mathbb R)$, it gives constants
$b,\beta>0$ and, for all large $\ell$, freely independent finite
sets $T_\ell\subset\Delta$ such that
\[
 \begin{gathered}
 |T_\ell|\ge e^{b\ell},\qquad
 T_\ell^{\pm1}\subset W_{\le6\ell}(S),\\
 \max_{t\in T_\ell^{\pm1}}\{\|t-I\|,\|\operatorname{Ad}(t)-I\|\}
 \le\epsilon_\ell:=e^{-\beta\ell}.
 \end{gathered}
\]
Let $\mu_\ell$ be uniform on $T_\ell\cup T_\ell^{-1}$ and let
$\overline\mu_\ell$ be its image in $\overline{G_1}$.
The standard free-group random-walk estimate bounds each atom of
$\mu_\ell^{*k}$ by $(\sqrt2e^{-b\ell/2})^k$.
Note that a cyclic subgroup of $\langle S\rangle$ meets the ball of radius $k$
for the generating set $T_\ell$ in at most $2k+1$ elements.
Lemma~\ref{lem:sl2-cyclic} therefore gives, uniformly over proper
connected closed subgroups $H<\overline{G_1}$,
\begin{equation}\label{eq:sl2-escape}
 \overline\mu_\ell^{*k}(H^{(e^{-6C\ell k})})
 \le(2k+1)(\sqrt2e^{-b\ell/2})^k
 \le e^{-b\ell k/4}
\end{equation}
for every $k\ge1$ and all large $\ell$.
\smallbreak

Thus there are $d_1,d_2>0$, independent of $\ell$, such that
for every fixed large $\ell$ and all sufficiently small $\delta$,
\begin{equation}\label{eq:sl2-bisg}
 \overline\mu_\ell^{*2n}(H^{(\delta)})\le\delta^{d_1},
 \qquad
 n=\left\lfloor
 d_2\frac{\log(1/\delta)}{\log(1/\epsilon_\ell)}
 \right\rfloor.
\end{equation}
For example, take $d_1=b/(96C)$ and $d_2=\beta/(24C)$.
This uniformity as $\epsilon_\ell\to0$ is the hypothesis of
\cite[Theorem~6.7]{BISG}, which gives restricted convolution
estimates tending to zero on every relatively compact measurable
set. Since the supports approach the identity,
\cite[Proposition~2.6]{BISG} gives local spectral gap for the image
of $\Delta$ in $\overline{G_1}$.
The covering argument in \cite[Section~2.4]{BISG} transfers the
restricted estimates to $\mathrm{SL}_2(\mathbb R)$, where
Proposition~2.6 applies again.
If the original projected factor is $\mathrm{PSL}_2(\mathbb R)$,
lift $\Gamma$ through the finite covering
$\mathrm{SL}_2(\mathbb R)\times G_2\to \mathrm{PSL}_2(\mathbb R)\times G_2$
and apply the conclusion for the projected image.
\end{proof}

\begin{proof}[Proof of Theorem~\ref{Cor: SL2 case}]
We first treat the center-free case. By the semisimple form of Borel
density for IRSs \cite[remark following Theorem~2.9]{7samurai},
$(G_x)^\circ$ is almost surely a fixed connected normal subgroup.
If it is nontrivial, it contains a simple factor, which acts trivially
almost everywhere. Irreducibility then forces the probability space
to be a point. Otherwise the stabilizers are discrete.
\smallbreak

Write $G=G_1\times G_2$, with $G_1=\mathrm{PSL}_2(\mathbb R)$.
By Lemma~\ref{lem:dense-projections}, either the stabilizers are
central almost surely or their projections to both $G_2$ and $G_1$
are dense. In the latter case Proposition~\ref{prop:sl2-projection}
supplies local spectral gap for $p_1(G_x)\curvearrowright G_1$.
Theorem~\ref{Theorem: A local spectral gap criterion} then makes
almost every stabilizer a lattice. The constants in the projection
argument may depend on $x$, as allowed by that theorem.
\smallbreak

For finite center, pass to the action of
$\overline G:=G/Z(G)$ on $\overline X:=X/Z(G)$.
It remains irreducible and $\overline G$ is a direct product of
center-free simple groups. For $\overline x$ the image of $x$,
\[
 \overline G_{\overline x}=G_xZ(G)/Z(G).
\]
If the quotient action is concentrated on a fixed point, the
original probability is supported on a finite central orbit.
Connectedness of $G$ and ergodicity then force a single fixed
point. If $\overline G_{\overline x}$ is trivial, $G_x$ is central.
If it is a lattice, its inverse image in $G$ is a lattice and
contains $G_x$ with index at most $|Z(G)|$. Hence $G_x$ is a
lattice as well.
\end{proof}

\end{document}